\documentclass[a4paper,12pt]{article}

\usepackage[utf8]{inputenc}
\usepackage[T1]{fontenc}
\usepackage[a4paper, margin=2.5cm]{geometry}

\usepackage{amsmath, amsthm, amssymb, enumerate}
\usepackage{graphicx}
\usepackage{enumerate}
\usepackage{authblk}
\usepackage[textsize=small, textwidth=2cm, color=yellow]{todonotes}
\usepackage{caption}
\usepackage[labelformat=simple]{subcaption}

\usepackage{inconsolata}
\usepackage{libertine}
\usepackage[absolute]{textpos}

\usepackage{comment}

\usepackage{thmtools}
\usepackage{thm-restate}

\usepackage{algpseudocode}

\usepackage[colorlinks=true, citecolor=red]{hyperref}

\usepackage{xfrac}
\usepackage{tikz}
\usetikzlibrary{shapes,snakes}
\usetikzlibrary{arrows,positioning} 
\tikzset{
    punkt/.style={
           circle,
           draw=black, very thick,
           text centered,
           inner sep = 2.5pt},
    dot/.style={
           circle,
           fill=black,
           inner sep=3pt
           },
    pil/.style={
           -,
           thick,
           },
     dir/.style={
           <-,
           thick,
           shorten <=2pt,
           shorten >=2pt,
           }
}

\declaretheorem[name=Theorem, numberwithin=section]{theorem}
\declaretheorem[name=Lemma, sibling=theorem]{lemma}

\declaretheorem[name=Definition, sibling=theorem]{definition}
\declaretheorem[name=Corollary, sibling=theorem]{corollary}
\declaretheorem[name=Conjecture, sibling=theorem]{conjecture}

\declaretheorem[name=Observation, sibling=theorem]{observation}
\declaretheorem[name=Question, style = remark, sibling=theorem]{question}

\renewcommand{\leq}{\leqslant}
\renewcommand{\geq}{\geqslant}

\def\cqedsymbol{\ifmmode$\lrcorner$\else{\unskip\nobreak\hfil
\penalty50\hskip1em\null\nobreak\hfil$\lrcorner$
\parfillskip=0pt\finalhyphendemerits=0\endgraf}\fi}

\title{On Vizing's edge colouring question\thanks{M.B., A.L. and J.N. are supported by ANR project GrR (\textsc{ANR-18-CE40-0032}). T. K. is supported by the grant no.~19-04113Y of the Czech Science Foundation (GA\v{C}R) and the Center for Foundations of Modern Computer Science (Charles Univ. project UNCE/SCI/004), her visit to Bordeaux, where part of the research was conducted, was supported by Czech-French Mobility project 8J19FR027.}}

\author[1]{Marthe Bonamy}
\author[2]{Oscar Defrain}
\author[3]{Tereza Klimo\v{s}ov\'{a}}
\author[2]{\\Aurélie Lagoutte}
\author[1]{Jonathan Narboni}

\affil[1]{CNRS, LaBRI, Université de Bordeaux, France.}
\affil[2]{Université Clermont Auvergne, CNRS, LIMOS UMR 6158, Clermont–Ferrand, France.}
\affil[3]{KAM, Charles University, Czech Republic.}

\begin{document}

\maketitle

\begin{abstract}
Soon after his 1964 seminal paper on edge colouring, Vizing asked the following question: can an optimal edge colouring be reached from any given proper edge colouring through a series of Kempe changes? We answer this question in the affirmative for triangle-free graphs.
\end{abstract}

\section{Introduction}\label{sec:intro}

Vizing proved in 1964~\cite{vizing1964estimate} that, to colour properly the edges of a simple graph, it suffices to have one more colour than the maximum number of neighbours.

\begin{theorem}[\cite{vizing1964estimate}]\label{th:Vizing}
Any simple graph $G$ satisfies $\chi'(G) \leq \Delta(G)+1$.
\end{theorem}
Here, $\chi'(G)$ denotes the chromatic index of $G$, that is, the smallest integer $k$ such that $G$ admits a proper $k$-edge-colouring. The largest degree of a vertex in $G$ is denoted by $\Delta(G)$, where the \emph{degree} of a vertex is the number of neighbours it has.

The proof of Theorem~\ref{th:Vizing} relies heavily on the notion of Kempe change, introduced in 1879 by Kempe\footnote{Initially in the context of vertex colouring.} in a failed attempt to prove the Four Colour Theorem~\cite{kempe1879geographical}. Given an edge-coloured graph, a \emph{Kempe chain} is a maximal connected bicoloured subgraph. A \emph{Kempe change} corresponds to selecting a Kempe chain and swapping the two colours in it. Observe that a Kempe chain may consist in a single edge $e$, coloured say $\alpha$, when some colour $\beta$ does not appears on any edge incident to $e$. A Kempe change on this Kempe chain therefore precisely recolour $e$ into $\beta$, possibly decreasing the total number of colours.

\begin{theorem}[\cite{vizing1964estimate}]\label{th:Vizingbis}
For every simple graph $G$, for any integer $k> \Delta(G)+1$, for any $k$-edge-colouring $\alpha$, there is a $(\Delta(G)+1)$-edge-colouring that can be reached from $\alpha$ through a series of Kempe changes.
\end{theorem}

Note that while some graphs need $\Delta(G)+1$ colours, some graphs can be edge-coloured with only $\Delta(G)$ colours. In the follow-up paper extending the result to multigraphs~\cite{vizing1965chromatic}, and later in a more publicly available survey paper~\cite{vizing1968some}, Vizing asks whether an optimal colouring can always be reached through a series of Kempe changes, as follows.

\begin{question}[\cite{vizing1965chromatic}]\label{qu:vizing}
For every simple graph $G$, for any integer $k> \chi'(G)$, for any $k$-edge-colouring $\alpha$, is there a $\chi'(G)$-edge-colouring that can be reached from $\alpha$ through a series of Kempe changes?
\end{question}

Question~\ref{qu:vizing} is in fact stated in the more general context of multigraphs.

Note that neither Theorem~\ref{th:Vizingbis} nor Question~\ref{qu:vizing} implies that all colourings with fewer colours are reachable, i.e., there is no choice regarding the target colouring. 
We say two $k$-edge-colourings are \emph{Kempe-equivalent} if one can be reached from the other through a series of Kempe changes using colours from $\{1,\ldots,k\}$. 
Question~\ref{qu:vizing}, if true, would imply~\cite{casselgren} and the following conjecture of Mohar~\cite{mohar2006kempe}, using the target $\chi'(G)$-colouring as an intermediate colouring.

\begin{conjecture}[\cite{mohar2006kempe}]\label{conj:Mohar}
For every simple graph $G$, all $(\Delta(G)+2)$-edge-colourings are Kempe-equivalent.
\end{conjecture}

Mohar proved the weaker case where $(\chi'(G)+2)$ colours are allowed.

\begin{theorem}[\cite{mohar2006kempe}]\label{th:mohar}
For every simple graph $G$, all $(\chi'(G)+2)$-edge-colourings are Kempe-equivalent.
\end{theorem}

As noted in~\cite{mcdonald2012kempe}, Theorem~\ref{th:mohar} is not true when replacing $(\chi'(G)+2)$ with $\chi'(G)$, regardless of whether $\chi'(G)=\Delta(G)$ (consider the graph $K_{5,5}$) or $\chi'(G)=\Delta(G)+1$ (consider the graph $K_5$). As noted in~\cite{mohar2006kempe}, it could however be true with $(\chi'(G)+1)$.

Not much is known towards Question~\ref{qu:vizing} or Conjecture~\ref{conj:Mohar}. In 2012, McDonald, Mohar and Scheide~\cite{mcdonald2012kempe} proved the case $\Delta(G)=3$ of the former (hence the case $\Delta(G)=4$ of the latter). In 2016, Asratian and Casselgren~\cite{casselgren} proved the case $\Delta(G)=4$ of the former (hence the case $\Delta(G)=5$ of the latter). We answer both questions affirmatively in the case where the graph is triangle-free, regardless of the value of $\Delta(G)$.

\begin{theorem}\label{th:main}
For every triangle-free graph $G$, for any integer $k>\chi'(G)$, any given $\chi'(G)$-edge-colouring can be reached from any $k$-edge-colouring through a series of Kempe changes.
\end{theorem}

Theorem~\ref{th:main} improves upon an earlier theorem concerning bipartite graphs~\cite{asratian2009note}. We derive the immediate following corollary.

\begin{corollary}\label{cor:main}
For every triangle-free graph $G$, all $(\chi'(G)+1)$-edge-colourings are Kempe equivalent.
\end{corollary}

The general approach toward Theorem~\ref{th:main} follows that of~\cite{casselgren}, which itself follows that of~\cite{mohar2006kempe}. From a $k$-edge-colouring with $k>\chi'(G)$, say we aim to reach a given $\chi'(G)$-colouring $\alpha$. We select a colour class $M$ of $\alpha$, and seek through a series of Kempe changes to reach a $k$-edge-colouring where $M$ is monochromatic and its colour appears on no other edge. We can then delete $M$ and apply induction on $\chi'(G)$.

\subsection*{Complexity implications}

As is often mentioned, Vizing's original argument can be turned into a polynomial-time algo\-rithm---this was formally noted by Misra and Gries in 1992~\cite{misra1992constructive}. However, deciding whether a graph $G$ is $\Delta(G)$-edge-colourable is an NP-complete problem~\cite{holyer1981np}, even in the case of triangle-free graphs~\cite{koreas1997np}. This leaves little hope for extracting a polynomial-time algorithm from the proof of Theorem~\ref{th:main}. There is however no difficulty in detecting the difference between Vizing's argument and ours: we start by assuming full access to a $\Delta(G)$-edge-colouring, which is crucial in the proof.

\subsection*{More about Kempe changes}

While this is irrelevant for the rest of the paper, let us mention some more applications and connections of Kempe changes to other problems. Since its introduction in the context of $4$-colouring planar graphs, much work has focused on determining which graph classes have good properties regarding Kempe-equivalence of their vertex colourings, see e.g.~\cite{mohar2006kempe} for a comprehensive overview or~\cite{bonamy2019conjecture} for a recent result on general graphs. We refer the curious reader to the relevant chapter of a 2013 survey by Cereceda~\cite{van2013complexity}. Kempe-equivalence falls within the wider setting of combinatorial reconfiguration, which~\cite{van2013complexity} is also an excellent introduction to.

Perhaps surprisingly, Kempe-equivalence has direct applications in approximate counting and statistical physics (see e.g.\ \cite{sokal2000personal,mohar2009new} for nice overviews). Closer to graph theory, Kempe-equivalence can be studied with a goal of obtaining a random colouring, by proving that a given random walk is a rapidly mixing Markov chain, see e.g.~\cite{vigoda}.

\subsection*{General setting of the proof}

Let us argue that it suffices to handle the case of a $\chi'(G)$-regular graph. 
Indeed, any graph $G$ is the induced subgraph of a $\chi'(G)$-regular graph that is also $\chi'(G)$-edge-colourable. To see this, we decrease step-by-step the difference between $\chi'(G)$ and the smallest degree of a vertex in $G$. Let $\beta$ be a $\chi'(G)$-edge-colouring of $G$, and consider two copies of $G$, each coloured $\beta$. We add an edge between both copies of every vertex of smallest degree: since both copies of $G$ are coloured the same, there is a colour available for the new edge. Note that this construction does not create any triangle. See Figure~\ref{fig:regular} for an example.
\begin{figure}[ht!]
\centering
\begin{subfigure}[t]{0.15\textwidth}
\centering
\begin{tikzpicture}[node distance=8mm and 8mm, auto]
 \node[dot](s){};
 \node[dot,right=of s](b){}
 edge[pil] node {2} (s);
 \node[dot,above =of b](a){}
 edge[pil,above] node {1} (s);
 \node[dot,below =of b](c){}
 edge[pil] node {3} (s);
  \node[below= of c] (empty) {};
\end{tikzpicture}
\caption{Graph $G$}
\end{subfigure}
\hfill
\begin{subfigure}[t]{0.25\textwidth}
\centering
\begin{tikzpicture}[node distance=8mm and 8mm, auto]
 \node[dot](s){};
 \node[dot,right=of s](b){}
 edge[pil] node {2} (s);
 \node[dot,above =of b](a){}
 edge[pil,above] node {1} (s);
 \node[dot,below =of b](c){}
 edge[pil] node {3} (s);
 
  \node[dot,right=of b](b2){}
  edge[pil] node {3} (b);
  \node[dot,right =of b2](s2){}
   edge[pil] node {2} (b2);
 \node[dot,above =of b2](a2){}
 edge[pil,above] node {1} (s2)
 edge[pil] node {2} (a);
 \node[dot,below =of b2](c2){}
 edge[pil,below] node {3} (s2)
 edge[pil,above] node {1} (c);
 \node[below= of c] (empty) {};
 
\end{tikzpicture}
\caption{Step 1 : minimum degree increases to 2}
\end{subfigure}
\hfill
\begin{subfigure}[t]{0.55\textwidth}
\centering
\begin{tikzpicture}[node distance=8mm and 8mm, auto]
 \node[dot](s){};
 \node[dot,right=of s](b){}
 edge[pil] node {2} (s);
 \node[dot,above =of b](a){}
 edge[pil,above] node {1} (s);
 \node[dot,below =of b](c){}
 edge[pil] node {3} (s);
 
  \node[dot,right=of b](b2){}
  edge[pil] node {3} (b);
  \node[dot,right =of b2](s2){}
   edge[pil, below] node {2} (b2);
 \node[dot,above =of b2](a2){}
 edge[pil,above] node {1} (s2)
 edge[pil] node {2} (a);
 \node[dot,below =of b2](c2){}
 edge[pil,below] node {3} (s2)
 edge[pil,above] node {1} (c);
 
  \node[dot, right=of s2](s3){};
 \node[dot,right=of s3](b3){}
 edge[pil, below] node {2} (s3);
 \node[dot,above =of b3](a3){}
 edge[pil,above] node {1} (s3);
 \node[dot,below =of b3](c3){}
 edge[pil] node {3} (s3);
 
  \node[dot,right=of b3](b4){}
  edge[pil] node {3} (b3);
  \node[dot,right =of b4](s4){}
   edge[pil] node {2} (b4);
 \node[dot,above =of b4](a4){}
 edge[pil,above] node {1} (s4)
 edge[pil] node {2} (a3);
 \node[dot,below =of b4](c4){}
 edge[pil,below] node {3} (s4)
 edge[pil,above] node {1} (c3);
  \node[below= of c] (empty) {};
 \path (a) edge[pil, bend left=30] node {3}  (a3) ;
 \path (a2) edge[pil, bend left=30] node {3}  (a4) ;
 \path (b) edge[pil, bend left=30] node {1}  (b3) ;
 \path (b2) edge[pil, bend left=30] node {1}  (b4) ;
 \path (c) edge[pil, bend right=30, below] node {2}  (c3) ;
 \path (c2) edge[pil, bend right=30,below] node {2}  (c4) ;
\end{tikzpicture}
\caption{Step 2 (final): a 3-regular graph}
\end{subfigure}
\caption{Construction of a $3$-regular $3$-edge colourable graph from a $3$-edge colourable graph}
\label{fig:regular}

\end{figure}
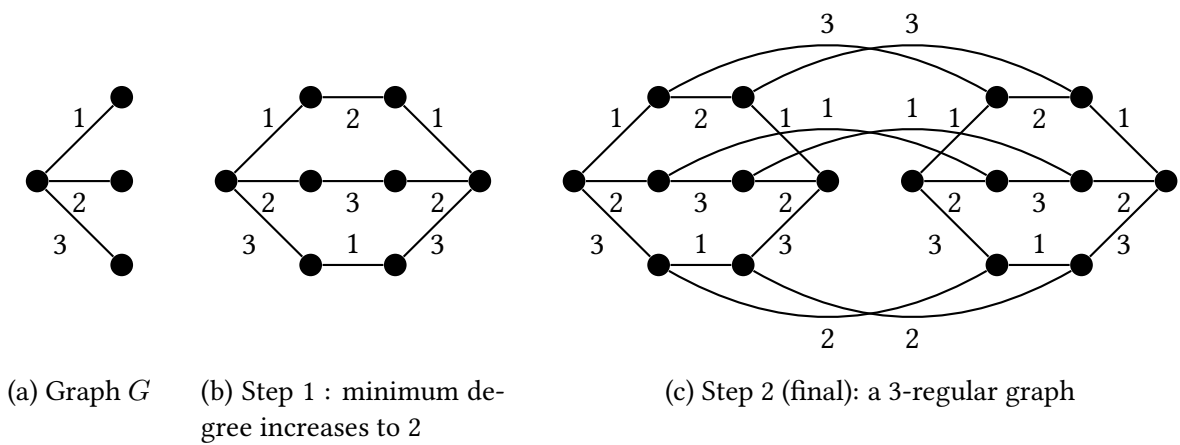

Additionally, note that any series of Kempe changes in a graph has a natural transposition to any induced subgraph of it. Indeed, if a Kempe chain in the graph corresponds to more than one Kempe chain in the induced subgraph, it suffices to operate the swap in every such Kempe chain.

This allows us to restrict our attention to the case where $\Delta(G)=\chi'(G)$ and the colour class $M$ is a perfect matching, which will prove to be convenient. Theorem~\ref{th:main} was already proved in~\cite{mohar2006kempe} when $k\geq \chi'(G)+2$. Therefore, we focus on the case $k=\chi'(G)+1$, though the reader can convince themself
that the proof could be adapted for higher $k$ with a loss in simplicity.

From now on, we consider only $(\Delta(G)+1)$-edge-colourings of a $\Delta(G)$-regular graph $G$. Therefore, for every such colouring $\alpha$, and for every vertex $u$, there is a unique colour $m_\alpha(u)$ in $ \{1,\ldots,\Delta+1\} \setminus \{\alpha(uv) \mid v \in N(u)\}$, referred to as the \emph{missing colour} of $\alpha$ at $u$.

We defined the notion of Kempe changes in the introduction: let us introduce some helpful notation around them. For any colouring $\alpha$, for any two (distinct) colours $c,d$, we denote by $K_\alpha(c,d)$ the subgraph of $G$ induced by the edges coloured $c$ or $d$. The notion of a component of $K_\alpha(c,d)$ containing an edge $e$ is straightforward. We extend this notion to that of a component containing a vertex $u$. 
To describe a Kempe change, we will indicate that we swap the component of $K_\alpha(c,d)$ containing this edge or that vertex, for some given $c$ and $d$. We will write $\alpha \leftrightsquigarrow \beta$ to indicate that two $k$-edge-colourings $\alpha$ and $\beta$ are Kempe-equivalent. Formally, we should indicate the bound on the number of colours involved in an intermediary colouring in the sequence of Kempe-changes. However, we believe that there is no ambiguity anywhere regarding this. In particular, throughout the proof we only involve colours in $\{1,\ldots,\Delta(G)+1\}$.

\section{Fan-like tools}\label{sec:fans}

Let $\alpha$ be a $(\Delta+1)$-edge-colouring of a $\Delta$-regular graph $G$. Consider an edge $uv$, and say we want to recolour it. If $m(u)=m(v)$, this can be done immediately without impacting the rest of the colouring. Therefore, let us consider $m(v) \neq m(u)$, and look at the obstacles around $u$. There is an edge $uw$ coloured $m(v)$. Again, if we can recolour it without impacting the rest of the colouring, we can then recolour $uv$ into $m(v)$. This prompts us to define a directed graph $D_u(\alpha)$ on vertex set $\{uw \mid w \in N(u)\}$, where a vertex $uw$ has a directed edge to $ux$ if $m(w)=\alpha(ux)$ (see Figure~\ref{fig:pcc}).
Note that by definition, every vertex in $D_u(\alpha)$ has out-degree $0$ or $1$, and arbitrarily large in-degree. 
Consider the sequence $X_u(\alpha,v)$ of vertices than can be reached from $uv$ in $D_u(\alpha)$. For both $D_u(\alpha)$ and $X_u(\alpha,v)$, we drop $\alpha$ from the notation when it is clear from context.

We have three possible scenarios, by increasing difficulty (see Figure~\ref{fig:pcc} for an illustration):
\begin{enumerate}
    \item\label{it:path} $X_u(v)$ induces a path in $D_u$.
    \item\label{it:cycle} $X_u(v)$ induces a cycle in $D_u$. 
    \item\label{it:comet} $X_u(v)$ induces a comet in $D_u$, where a \emph{comet} is obtained from a directed path by adding an edge from the sink to a vertex that is neither the source nor the sink.
\end{enumerate}

\begin{figure}[!ht]
    \hspace{0.15cm}
\begin{tikzpicture}[node distance=7mm and 7mm, auto]
 nodes
  \node[punkt,label={-90:$u$}](u){1};
  \node[punkt, above right=of u, label={-45:$x_0$}](x0){3}
  	edge[pil] node {2}(u);
  \node[punkt, below right=of u, label={-45:$x_1$}](x1){4}
  	edge[pil] node {3}(u);
  \node[punkt, below left=of u, label={-135:$x_2$}](x2){5}
  	edge[pil] node {4}(u);
  \node[punkt, above left=of u, label={-135:$x_3$}](x3){1}
  	edge[pil] node {5}(u);
\end{tikzpicture}
  \hfill
\begin{tikzpicture}[node distance=7mm and 7mm, auto]
 nodes
  \node[punkt,label={-90:$u$}](u){1};
  \node[punkt, above right=of u, label={-45:$x_0$}](x0){3}
  	edge[pil] node {2}(u);
  \node[punkt, below right=of u, label={-45:$x_1$}](x1){4}
  	edge[pil] node {3}(u);
  \node[punkt, below left=of u, label={-135:$x_2$}](x2){5}
  	edge[pil] node {4}(u);
  \node[punkt, above left=of u, label={-135:$x_3$}](x3){2}
  	edge[pil] node {5}(u);
\end{tikzpicture}
  \hfill
\begin{tikzpicture}[node distance=7mm and 7mm, auto]
 nodes
  \node[punkt,label={-90:$u$}](u){1};
  \node[punkt, above right=of u, label={-45:$x_0$}](x0){3}
  	edge[pil] node {2}(u);
  \node[punkt, below right=of u, label={-45:$x_1$}](x1){4}
  	edge[pil] node {3}(u);
  \node[punkt, below left=of u, label={-135:$x_2$}](x2){5}
  	edge[pil] node {4}(u);
  \node[punkt, above left=of u, label={-135:$x_3$}](x3){3}
  	edge[pil] node {5}(u);
\end{tikzpicture}
    \hspace{0.2cm}
    \hfill
  
 \vspace{8mm}
 
   \hspace{0.9cm}
  \begin{tikzpicture}[node distance=1.3cm and 1.3cm, auto]
 nodes
  \node(ux1) {$ux_0$};
  \node[below=of ux1](ux2){$ux_1$}
    edge[dir] (ux1);
  \node[left=of ux2](ux3){$ux_2$}
    edge[dir] (ux2);
  \node[above=of ux3](ux4){$ux_3$}
    edge[dir] (ux3);
  \end{tikzpicture}
  \hfill
   \begin{tikzpicture}[node distance=1.3cm and 1.3cm, auto]
 nodes
  \node(ux1) {$ux_0$};
  \node[below=of ux1](ux2){$ux_1$}
    edge[dir] (ux1);
  \node[left=of ux2](ux3){$ux_2$}
    edge[dir] (ux2);
  \node[above=of ux3](ux4){$ux_3$}
    edge[dir] (ux3);
  \path (ux1) edge[dir] (ux4);
 
 \end{tikzpicture}
 \hfill
  \begin{tikzpicture}[node distance=1.3cm and 1.3cm, auto]
 nodes
  \node(ux1) {$ux_0$};
  \node[below=of ux1](ux2){$ux_1$}
    edge[dir] (ux1);
  \node[left=of ux2](ux3){$ux_2$}
    edge[dir] (ux2);
  \node[above=of ux3](ux4){$ux_3$}
    edge[dir] (ux3);
  \path (ux2) edge[dir] (ux4);
 
 \end{tikzpicture}
    \hspace{0.9cm}
    \hfill
   
 \caption{From left to right, the three possible scenarios for a sequence $X_{u}(\alpha,x_0)$ in the digraph $D_{u}(\alpha)$: a path, a cycle or a comet. (Vertices are labeled by the missing colors.)}
 \label{fig:pcc}
\end{figure}
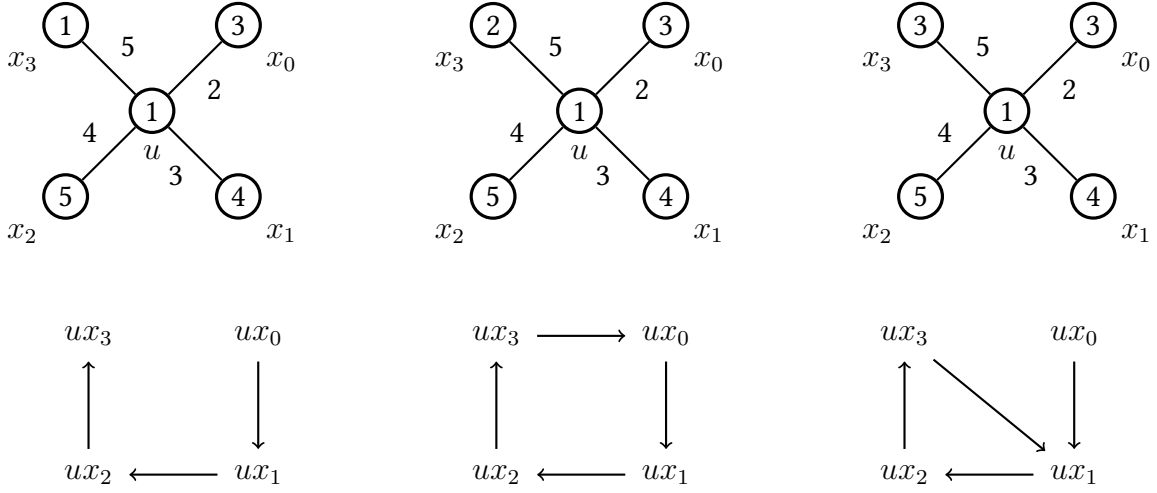

For any edge $uv$, if $X_u(\alpha,v)$ induces a path or cycle in $D_u$, we denote by $X_{u}^{-1}(\alpha,v)$ the colouring obtained from $\alpha$ by assigning the colour $m(w)$ to every edge $uw \in X_u(\alpha,v)$. 
Note that for every edge $uw \in X_u(\alpha,v)$, we have $m_{X_u^{-1}(\alpha,v)}(w)=\alpha(uw)$. 
We refer to this operation on $\alpha$ as \emph{inverting} $X_u(\alpha,v)$. Figure~\ref{fig:path-swap} illustrates the result of inverting a path. We drop $v$ from the notation when there is no ambiguity. 

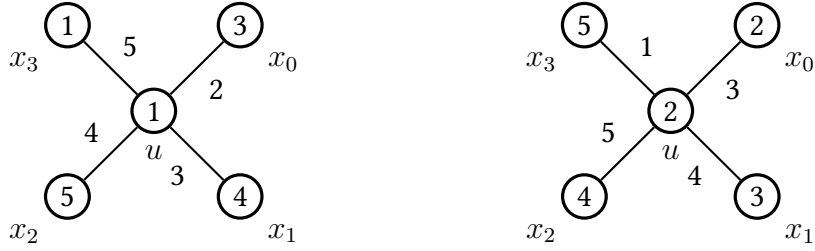
\begin{figure}[!ht]
\begin{center}

\begin{tikzpicture}[node distance=7mm and 7mm, auto]
 nodes
  \node[punkt,label={-90:$u$}](u){1};
  \node[punkt, above right=of u, label={-45:$x_0$}](x0){3}
    edge[pil] node {2}(u);
  \node[punkt, below right=of u, label={-45:$x_1$}](x1){4}
    edge[pil] node {3}(u);
  \node[punkt, below left=of u, label={-135:$x_2$}](x2){5}
    edge[pil] node {4}(u);
  \node[punkt, above left=of u, label={-135:$x_3$}](x3){1}
    edge[pil] node {5}(u);
\end{tikzpicture}
\hspace{25mm}
\begin{tikzpicture}[node distance=7mm and 7mm, auto]
 nodes
  \node[punkt,label={-90:$u$}](u){2};
  \node[punkt, above right=of u, label={-45:$x_0$}](x0){2}
    edge[pil] node {3}(u);
  \node[punkt, below right=of u, label={-45:$x_1$}](x1){3}
    edge[pil] node {4}(u);
  \node[punkt, below left=of u, label={-135:$x_2$}](x2){4}
    edge[pil] node {5}(u);
  \node[punkt, above left=of u, label={-135:$x_3$}](x3){5}
    edge[pil] node {1}(u);
\end{tikzpicture}

 \end{center}
\caption{Coloring $\alpha$ (left) and $X_{u}^{-1}(\alpha,x_0)$ (right) when $X_{u}(\alpha,x_0)$ is a path.}
\label{fig:path-swap}
\end{figure}

In order to have an overview of the key ingredients in the proof, let us now state an Observation and some Lemmas, the proof of which are postponed to the following subsections.

\begin{observation}\label{obs:path}
For any vertex $u$ and path $X_u(\alpha)$ in $D_u(\alpha)$, $\alpha \leftrightsquigarrow X_u^{-1}(\alpha)$. 
\end{observation}

\begin{definition}\label{def:saturated}
For any vertex $u$ and cycle $X_u(\alpha)=(ux_0,\ldots,ux_p)$ in $D_u(\alpha)$, we say that $X_u(\alpha)$ is \emph{saturated} if for every $0\leq i\leq p$, the component of $K(\alpha(ux_i),m(u))$ containing $u$ also contains $x_{i-1}$ (resp.~$x_p$ if $i=0$).
\end{definition}

The same conclusion holds for cycles unless the sequence is saturated:

\begin{lemma}\label{lem:loop}
For any vertex $u$ and non-saturated cycle $X_u(\alpha)$ in $D_u(\alpha)$, $\alpha \leftrightsquigarrow X_u^{-1}(\alpha)$. 
\end{lemma}

To reach the desired conclusion for a saturated cycle, we need further assumptions, including the absence of triangles, as follows:

\begin{lemma}\label{lem:doubleloop}
For any vertex $u$ and saturated cycle $X_u(\alpha,v)$ in $D_u(\alpha)$, if $G$ is triangle-free, and
if the sequence $Y_v(\alpha,u)$ of vertices of $D_v(\alpha)$ induces a cycle, then $\alpha \leftrightsquigarrow X_u^{-1}(\alpha,v)$. 
\end{lemma}

For comets, it suffices to allow one Kempe change outside of $X_u(\alpha)$:

\begin{lemma}\label{lem:comet}
For any vertex $u$ and comet $X_u(\alpha)=(ux_0,\ldots,ux_p)$ in $D_u(\alpha)$, we have $\alpha \leftrightsquigarrow \alpha' $, where $\alpha'$ satisfies $m_{\alpha'}(u)=\alpha(ux_0)$ and is obtained from $\alpha$ by changing the colour of some edges in $X_u(\alpha)$ and possibly swapping one component $C$ in $K(m(u),\alpha(ux_q))$, where $ux_q$ is the endpoint of the out-edge from $ux_p$ in $D_u(\alpha)$. 
\end{lemma}

In the colouring $\alpha'$ obtained from Lemma~\ref{lem:comet}, we stress the fact that the number of edges coloured $\alpha(ux_0)$ strictly decreases 
as the swapped component $C$ does not contain such a colour, and $m_{\alpha'}(u)=\alpha(ux_0)$, i.e., no edge incident to $u$ has colour $\alpha(ux_0)$ in $\alpha'$.

\medskip

We prove the lemmas by increasing difficulty in the following subsections.

\subsection*{Gentle introduction: a proof of Observation~\ref{obs:path}}

\begin{proof}[Proof of Observation~\ref{obs:path}] 
Let $X_u(\alpha)=(ux_0,\ldots,ux_p)$ be a path in $D_u(\alpha)$. 
Intuitively, we will start recolouring edges from the end of the path to its beginning.
Observe that since $X_u(\alpha)$ is a path, by construction of $D_u(\alpha)$ there is no edge incident to $u$ that has colour $m(x_p)$, hence $m(u)=m(x_p)$. 
We proceed by induction on $p$. 
When $p=0$, we have $m(x_0)=m(u)$, thus swapping the single-edge component of $K(\alpha(ux_0),m(u))$ containing $ux_0$ yields the desired colouring $X_u^{-1}(\alpha)$.

Similarly, for $p>0$, we swap the (single-edge) component of $K(\alpha(ux_p),m(u))$ containing $ux_p$, and denote by $\alpha'$ the resulting colouring. We note that in $D_u(\alpha')$, the sequence $X_{u}(\alpha',x_0)$ is exactly the path $(ux_0,\ldots,ux_{p-1})$. Moreover $X_{u}(\alpha',x_0)=X_u^{-1}(\alpha)$.
By induction we derive $\alpha' \leftrightsquigarrow X^{-1}_u(\alpha)$, hence $\alpha \leftrightsquigarrow X_u^{-1}(\alpha)$.
\end{proof}

\subsection*{Comets: a proof of Lemma~\ref{lem:comet}}

\begin{proof}[Proof of Lemma~\ref{lem:comet}]
Let $X_u(\alpha)=(ux_0,\dots,ux_p)$ be a comet in $D_u(\alpha)$, with $x_q$ the endpoint of the out-edge from $ux_p$ in $D_u(\alpha)$. 
Since $X_u(\alpha)$ is a comet, $0<q<p$. 
We swap the component $C$ of $K(m(u),\alpha(ux_q))$ containing the edge $ux_q$, and denote by $\alpha'$ the resulting colouring. In $\alpha$, we have $m_\alpha(x_p)=m_\alpha(x_{q-1})=\alpha(ux_q)$. Since $C$ must be a path, it contains at most two vertices (its endpoints) whose missing colour in $\alpha$ belongs to $\{m(u),\alpha(ux_q)\}$.
We know that $C$ already contains $u$, so at least one of $x_p$ and $x_{q-1}$ has the same missing colour in $\alpha$ and $\alpha'$. We distinguish the two cases.

\begin{itemize}
    \item Assume $m_{\alpha'}(x_{q-1})=\alpha(ux_q)$. Since $m_{\alpha'}(u)=\alpha(ux_q)$, it follows that in $D_u(\alpha')$, the sequence $X_u(\alpha',x_0)$ is exactly $(ux_0,\ldots,ux_{q-1})$, which induces a path. We then conclude by Observation~\ref{obs:path}.
    \item If not, $m_{\alpha'}(x_{q-1})=\alpha'(ux_q)$, and $m_{\alpha'}(x_p)=\alpha(ux_q)$. Since $m_{\alpha'}(u)=\alpha(ux_q)$, it follows that in $D_u(\alpha')$, the sequence $X_u(\alpha',x_0)$ is exactly $(ux_0,\ldots,ux_{p})$, which induces a path. We then conclude by Observation~\ref{obs:path}.
\end{itemize}%
\end{proof}
\subsection*{Non-saturated cycles: a proof of Lemma~\ref{lem:loop}}

\begin{figure}
\centering

\begin{subfigure}[t]{0.45\textwidth}
\centering
\begin{tikzpicture}[node distance=7mm and 7mm, auto]
 nodes
  \node[punkt,label={-90:$u$}](u){1};
  \node[punkt, above right=of u, label={-45:$x_0$}](x1){3}
  	edge[pil] node {2}(u);
  \node[punkt, below right=of u, label={-45:$x_1$}](x2){4}
  	edge[pil] node {3}(u);
  \node[punkt, below left=of u, label={-135:$x_2$}](x3){5}
  	edge[pil] node {4}(u);
  \node[punkt, above left=of u, label={-135:$x_3$}](x4){2}
  	edge[pil] node {5}(u);
  	
 \node[left = of x4,rectangle, draw=black] {1-2 chain}
   edge[pil,above] node {1} (x4);
\end{tikzpicture}
\caption{Colouring $\alpha$}\label{fig:nonsat-cyclea}
\end{subfigure} 
\hfill
\begin{subfigure}[t]{0.45\textwidth}
\centering
\begin{tikzpicture}[node distance=7mm and 7mm, auto]
 nodes
  \node[punkt,label={-90:$u$}](u){1};
  \node[punkt, above right=of u, label={-45:$x_0$}](x1){3}
  	edge[pil] node {2}(u);
  \node[punkt, below right=of u, label={-45:$x_1$}](x2){4}
  	edge[pil] node {3}(u);
  \node[punkt, below left=of u, label={-135:$x_2$}](x3){5}
  	edge[pil] node {4}(u);
  \node[punkt, above left=of u, label={-135:$x_3$}](x4){1}
  	edge[pil] node {5}(u);
  	
 \node[left = of x4,rectangle, draw=black] {2-1 chain}
   edge[pil,above] node {2} (x4);
\end{tikzpicture}
\caption{Colouring $\alpha'$ }\label{fig:nonsat-cycleb}
\end{subfigure} \hfill
\vspace{5mm}

\begin{subfigure}[t]{0.45\textwidth} 
\centering
\begin{tikzpicture}[node distance=7mm and 7mm, auto]
 nodes
  \node[punkt,label={-90:$u$}](u){2};
  \node[punkt, above right=of u, label={-45:$x_0$}](x1){2}
  	edge[pil] node {3}(u);
  \node[punkt, below right=of u, label={-45:$x_1$}](x2){3}
  	edge[pil] node {4}(u);
  \node[punkt, below left=of u, label={-135:$x_2$}](x3){4}
  	edge[pil] node {5}(u);
  \node[punkt, above left=of u, label={-135:$x_3$}](x4){5}
  	edge[pil] node {1}(u);
  	
 \node[left = of x4,rectangle, draw=black] {2-1 chain}
   edge[pil,above] node {2} (x4);
\end{tikzpicture}
\caption{Colouring $\alpha''=X_{u}^{-1}(\alpha',x_0)$}\label{fig:nonsat-cyclec}
\end{subfigure}
\hfill
\begin{subfigure}[t]{0.45\textwidth} 
\centering
\hspace{-1.45cm}
\begin{tikzpicture}[node distance=7mm and 7mm, auto]
 nodes
  \node[punkt,label={-90:$u$}](u){1};
  \node[punkt, above right=of u, label={-45:$x_0$}](x1){1}
  	edge[pil] node {3}(u);
  \node[punkt, below right=of u, label={-45:$x_1$}](x2){3}
  	edge[pil] node {4}(u);
  \node[punkt, below left=of u, label={-135:$x_2$}](x3){4}
  	edge[pil] node {5}(u);
  \node[punkt, above left=of u, label={-135:$x_3$}](x4){5}
  	edge[pil] node {2}(u);
  	
 \node[left = of x4,rectangle, draw=black] {1-2 chain}
   edge[pil,above] node {1} (x4);
\end{tikzpicture}
\caption{Colouring $X_{u}^{-1}(\alpha,x_0)$}\label{fig:nonsat-cycled}
\end{subfigure} \hfill

\caption{Colorings $\alpha \leftrightsquigarrow \alpha' \leftrightsquigarrow X_{u}^{-1}(\alpha',x_0) \leftrightsquigarrow X_{u}^{-1}(\alpha,x_0)$.}
\label{fig:nonsat-cycle}
\end{figure}
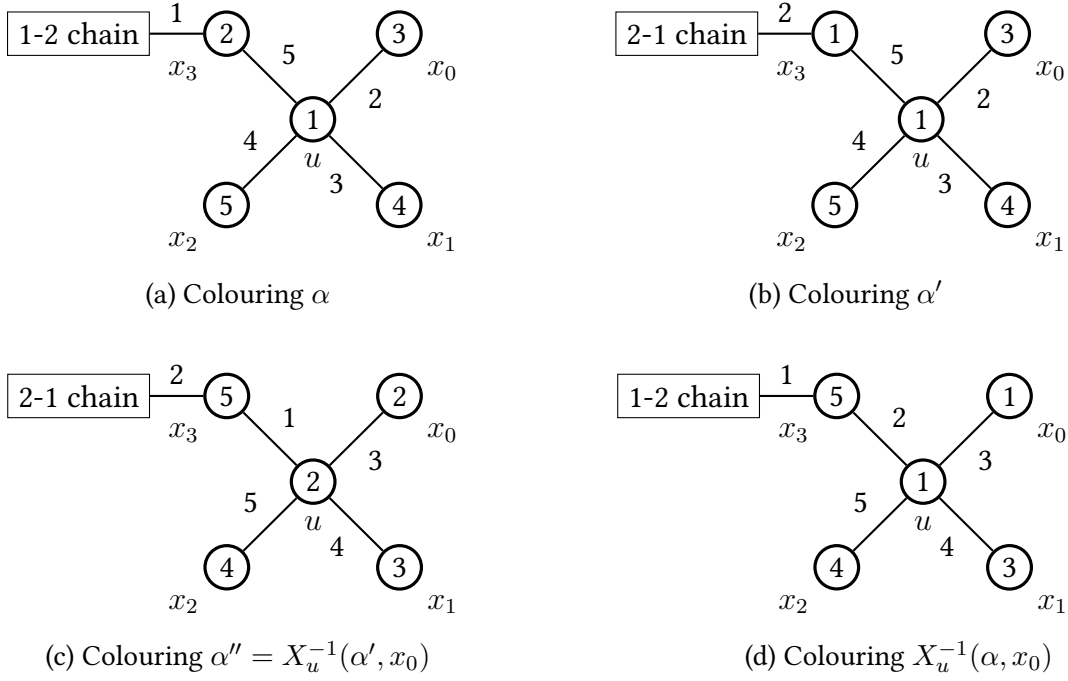

\begin{proof}[Proof of Lemma~\ref{lem:loop}] 
Let $X_u(\alpha)=(ux_0,\dots,ux_p)$ be a non-saturated cycle in $D_u(\alpha)$.
Without loss of generality since $X_u(\alpha)$ induces a cycle that 
is not saturated,
we can assume that the component of $K(\alpha(ux_0),m(u))$ containing $u$ does not contain an edge incident with $x_{p}$. 
By definition of $D_u(\alpha)$, we have $m(x_i)=\alpha(ux_{i+1})$ for every $0\leq i<p$, and as illustrated on Figure~\ref{fig:nonsat-cyclea}
$m(x_p)=\alpha(ux_0)$. We consider the colouring $ \alpha'$ obtained from $\alpha$ by swapping the component $C$ of $K(\alpha(ux_0),m(u))$ containing $x_p$ ($C$ is referred to as a 1-2 chain on Figure~\ref{fig:nonsat-cyclea}, see Figure~\ref{fig:nonsat-cycleb} for the resulting colouring). 
By assumption, this has no impact on the colours of the edges incident with $u$, and $m_\alpha(x_i)=m_{\alpha'}(x_i)$ for every $0\leq i<p$, as well as $m_\alpha(u)=m_{\alpha'}(u)$. 
Note however that $m_{\alpha'}(x_p)=m_\alpha(u)$. 
In the colouring $\alpha'$, $X_u(\alpha',x_0) = (ux_0,...,ux_p)$ is a path, thus by Observation~\ref{obs:path}, $\alpha' \leftrightsquigarrow X_u^{-1}(\alpha',x_0)$; we denote this resulting colouring by $\alpha''$ (see Figure~\ref{fig:nonsat-cyclec}). 
In the colouring $\alpha''$, let $C'$ be the component of $K(\alpha(ux_0),m_{\alpha}(u))$ containing $x_p$. We have that $C'=C\cup \{ux_p\}$, and that $m_{\alpha''}(u) = \alpha(ux_0)$, so it suffices to swap $C'$ to obtain $X_u^{-1}(\alpha)$ as illustrated on Figure~\ref{fig:nonsat-cycled}. Hence $\alpha \leftrightsquigarrow \alpha' \leftrightsquigarrow X_u^{-1}(\alpha',x_0) \leftrightsquigarrow X_u^{-1}(\alpha)$, as desired.
\end{proof}

\subsection*{Double cycles: a proof of Lemma~\ref{lem:doubleloop}}

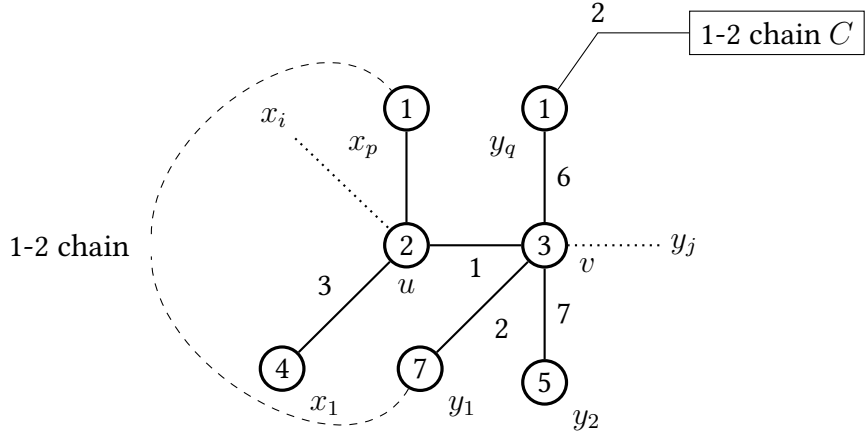
\begin{figure}
\centering
\hspace{-3cm}


%
\centering

\begin{tikzpicture}[node distance=12mm and 12mm, auto]
 nodes
  \node[punkt,label={-90:$u$}](u){2};
  \node[punkt,label={-4:$v$},right=of u](v){3}
  edge[pil] node {1} (u);
  \node[punkt, above=of u, label={-135:$x_p$}](xp){1}
  	edge[pil] (u);
  \node[punkt, above=of v, label={-135:$y_q$}](yq){1}
  	edge[pil] node {6} (v)  ;
  \node[punkt, below left =of v,  label={-45:$y_1$}](yi){7}
  	edge[pil, below right] node {2} (v);
  \node[left=29mm of u, label={180:1-2 chain}](ch){}; 
  \draw[dashed](yi)to[out=-120,in=-90](ch)to[out=90,in=130](xp);

  \node[coordinate,shift={(.7,1)}, label={2}](brisure) at (yq) {2};
  \node[right = of brisure, rectangle, draw=black](C){1-2 chain $C$};
   \draw[-](yq)--(brisure)|-(C);
   \node[punkt, below=of v, label={-45:$y_2$}](y2){5}
  	edge[pil]  node[right] {7} (v);
  	\node[punkt, below left=of u, label={-45:$x_1$}](x1){4}
  	edge[pil]  node{3}(u);
  	\node[ above left=of u, ](xi){$x_i$}
  	edge[pil, dotted]  (u);
  	\node[ right=of v, ](yj){$y_j$}
  	edge[pil, dotted]  (v);
\end{tikzpicture}

\caption{Double cycles: Illustration of the beginning of the proof of Lemma~\ref{lem:doubleloop}: In colouring~$\alpha$, the vertex $y_q$ is in a different component of $K(\alpha(uv),m(u))$ than $u$ and $x_p$.}
\label{fig:double-cycle-begin}
\end{figure}

\begin{figure}
\begin{subfigure}[t]{0.45\textwidth}
\centering

\begin{tikzpicture}[node distance=12mm and 12mm, auto]
 nodes
  \node[punkt,label={-90:$u$}](u){2};
  \node[punkt,label={-4:$v$},right=of u](v){3}
  edge[pil] node {1} (u);
  \node[punkt, above=of u, label={-135:$x_p$}](xp){1}
  	edge[pil] (u);
  \node[punkt, above=of v, label={-135:$y_q$}](yq){2}
  	edge[pil] node {6} (v)  ;
  \node[punkt, below left =of v,  label={-45:$y_1$}](yi){7}
  	edge[pil, below right] node {2} (v);
  \node[coordinate,shift={(.7,1)}, label={1}](brisure) at (yq) {1};
  \node[right = of brisure, rectangle, draw=black](C){2-1 chain $C$};
   \draw[-](yq)--(brisure)|-(C);
   \node[right = of yq, rectangle, draw=black](Cp){2-3 chain $C'$}
   edge[pil] node {3} (yq);
   \node[punkt, below=of v, label={-45:$y_2$}](y2){5}
  	edge[pil]  node[right] {7} (v);
  	\node[punkt, below left=of u, label={-45:$x_1$}](x1){4}
  	edge[pil]  node{3}(u);
  	\node[ above left=of u, ](xi){$x_i$}
  	edge[pil, dotted]  (u);
  	\node[ right=of v, ](yj){$y_j$}
  	edge[pil, dotted]  (v);
\end{tikzpicture}
\subcaption{$\alpha_1$}
\end{subfigure}
\hspace{0.07\textwidth}
\begin{subfigure}[t]{0.45\textwidth}
\centering

\begin{tikzpicture}[node distance=12mm and 12mm, auto]
 nodes
  \node[punkt,label={-90:$u$}](u){2};
  \node[punkt,label={-4:$v$},right=of u](v){3}
  edge[pil] node {1} (u);
  \node[punkt, above=of u, label={-135:$x_p$}](xp){1}
  	edge[pil] (u);
  \node[punkt, above=of v, label={-135:$y_q$}](yq){3}
  	edge[pil] node {6} (v)  ;
  \node[punkt, below left =of v,  label={-45:$y_1$}](yi){7}
  	edge[pil, below right] node {2} (v);
  \node[coordinate,shift={(.7,1)}, label={1}](brisure) at (yq) {1};
  \node[right = of brisure, rectangle, draw=black](C){2-1 chain $C$};
   \draw[-](yq)--(brisure)|-(C);
   \node[right = of yq, rectangle, draw=black](Cp){3-2 chain $C'$}
   edge[pil] node {2} (yq);
   \node[punkt, below=of v, label={-45:$y_2$}](y2){5}
  	edge[pil]  node[right] {7} (v);
  	\node[punkt, below left=of u, label={-45:$x_1$}](x1){4}
  	edge[pil]  node{3}(u);
  	\node[ above left=of u, ](xi){$x_i$}
  	edge[pil, dotted]  (u);
  	\node[ right=of v, ](yj){$y_j$}
  	edge[pil, dotted]  (v);
\end{tikzpicture}
\subcaption{$\alpha_2$}
\end{subfigure}

\begin{subfigure}[t]{0.45\textwidth}
\centering
\begin{tikzpicture}[node distance=12mm and 12mm, auto]
 nodes
  \node[punkt,label={-90:$u$}](u){1};
  \node[punkt,label={-4:$v$},right=of u](v){1}
  edge[pil] node {2} (u);
  \node[punkt, above=of u, label={-135:$x_p$}](xp){1}
  	edge[pil] (u);
  \node[punkt, above=of v, label={-135:$y_q$}](yq){6}
  	edge[pil] node {3} (v)  ;
  \node[punkt, below left =of v,  label={-45:$y_1$}](yi){2}
  	edge[pil, below right] node {7} (v);
  \node[coordinate,shift={(.7,1)}, label={1}](brisure) at (yq) {1};
  \node[right = of brisure, rectangle, draw=black](C){2-1 chain $C$};
   \draw[-](yq)--(brisure)|-(C);
   \node[right = of yq, rectangle, draw=black](Cp){3-2 chain $C'$}
   edge[pil] node {2} (yq);
   \node[punkt, below=of v, label={-45:$y_2$}](y2){7}
  	edge[pil]  node[right] {5} (v);
  	\node[punkt, below left=of u, label={-45:$x_1$}](x1){4}
  	edge[pil]  node{3} (u);
  	\node[ above left=of u, ](xi){$x_i$}
  	edge[pil, dotted]  (u);
  	\node[ right=of v, ](yj){$y_j$}
  	edge[pil, dotted]  (v);
\end{tikzpicture}
\subcaption{$\alpha_3$}
\end{subfigure}
\hspace{0.07\textwidth}
\begin{subfigure}[t]{0.45\textwidth}
\centering
\begin{tikzpicture}[node distance=12mm and 12mm, auto]
 nodes
  \node[punkt,label={-90:$u$}](u){3};
  \node[punkt,label={-4:$v$},right=of u](v){1}
  edge[pil] node {2} (u);
  \node[punkt, above=of u, label={-135:$x_p$}](xp){{\color{white} 8}}
  	edge[pil] node {1} (u);
  \node[punkt, above=of v, label={-135:$y_q$}](yq){6}
  	edge[pil] node {3} (v)  ;
  \node[punkt, below left =of v,  label={-45:$y_1$}](yi){2}
  	edge[pil, below right] node {7} (v);
  \node[coordinate,shift={(.7,1)}, label={1}](brisure) at (yq) {1};
  \node[right = of brisure, rectangle, draw=black](C){2-1 chain $C$};
   \draw[-](yq)--(brisure)|-(C);
   \node[right = of yq, rectangle, draw=black](Cp){3-2 chain $C'$}
   edge[pil] node {2} (yq);
   \node[punkt, below=of v, label={-45:$y_2$}](y2){7}
  	edge[pil]  node[right] {5} (v);
  	\node[punkt, below left=of u, label={-45:$x_1$}](x1){3}
  	edge[pil]  node{4} (u);
  	\node[ above left=of u, ](xi){$x_i$}
  	edge[pil, dotted]  (u);
  	\node[ right=of v, ](yj){$y_j$}
  	edge[pil, dotted]  (v);
\end{tikzpicture}
\subcaption{$\alpha_4$}
\end{subfigure}

\begin{subfigure}[t]{0.45\textwidth}
\centering
\begin{tikzpicture}[node distance=12mm and 12mm, auto]
 nodes
  \node[punkt,label={-90:$u$}](u){2};
  \node[punkt,label={-4:$v$},right=of u](v){1}
  edge[pil] node {3} (u);
  \node[punkt, above=of u, label={-135:$x_p$}](xp){{\color{white} 8}}
  	edge[pil] node {1} (u);
  \node[punkt, above=of v, label={-135:$y_q$}](yq){6}
  	edge[pil] node {2} (v)  ;
  \node[punkt, below left =of v,  label={-45:$y_1$}](yi){2}
  	edge[pil, below right] node {7} (v);
  \node[coordinate,shift={(.7,1)}, label={1}](brisure) at (yq) {1};
  \node[right = of brisure, rectangle, draw=black](C){2-1 chain $C$};
   \draw[-](yq)--(brisure)|-(C);
   \node[right = of yq, rectangle, draw=black](Cp){2-3 chain $C'$}
   edge[pil] node {3} (yq);
   \node[punkt, below=of v, label={-45:$y_2$}](y2){7}
  	edge[pil]  node[right] {5} (v);
  	\node[punkt, below left=of u, label={-45:$x_1$}](x1){3}
  	edge[pil]  node{4} (u);
  	\node[ above left=of u, ](xi){$x_i$}
  	edge[pil, dotted]  (u);
  	\node[ right=of v, ](yj){$y_j$}
  	edge[pil, dotted]  (v);
\end{tikzpicture}
\subcaption{$\alpha_5$}
\end{subfigure}
\hspace{0.07\textwidth}
\begin{subfigure}[t]{0.45\textwidth}
\centering
\begin{tikzpicture}[node distance=12mm and 12mm, auto]
 nodes
  \node[punkt,label={-90:$u$}](u){2};
  \node[punkt,label={-4:$v$},right=of u](v){1}
  edge[pil] node {3} (u);
  \node[punkt, above=of u, label={-135:$x_p$}](xp){{\color{white} 8}}
  	edge[pil] node {1} (u);
  \node[punkt, above=of v, label={-135:$y_q$}](yq){2}
  	edge[pil] node {6} (v)  ;
  \node[punkt, below left =of v,  label={-45:$y_1$}](yi){7}
  	edge[pil, below right] node {2} (v);
  \node[coordinate,shift={(.7,1)}, label={1}](brisure) at (yq) {1};
  \node[right = of brisure, rectangle, draw=black](C){2-1 chain $C$};
   \draw[-](yq)--(brisure)|-(C);
   \node[right = of yq, rectangle, draw=black](Cp){2-3 chain $C'$}
   edge[pil] node {3} (yq);
   \node[punkt, below=of v, label={-45:$y_2$}](y2){5}
  	edge[pil]  node[right] {7} (v);
  	\node[punkt, below left=of u, label={-45:$x_1$}](x1){3}
  	edge[pil]  node{4} (u);
  	\node[ above left=of u, ](xi){$x_i$}
  	edge[pil, dotted]  (u);
  	\node[ right=of v, ](yj){$y_j$}
  	edge[pil, dotted]  (v);
\end{tikzpicture}
\subcaption{$\alpha_6$}
\end{subfigure}

\caption{Double cycles : illustration of the intermediate colouring in the proof of Lemma~\ref{lem:doubleloop}.}
\label{fig:double-cycle-end}
\end{figure}
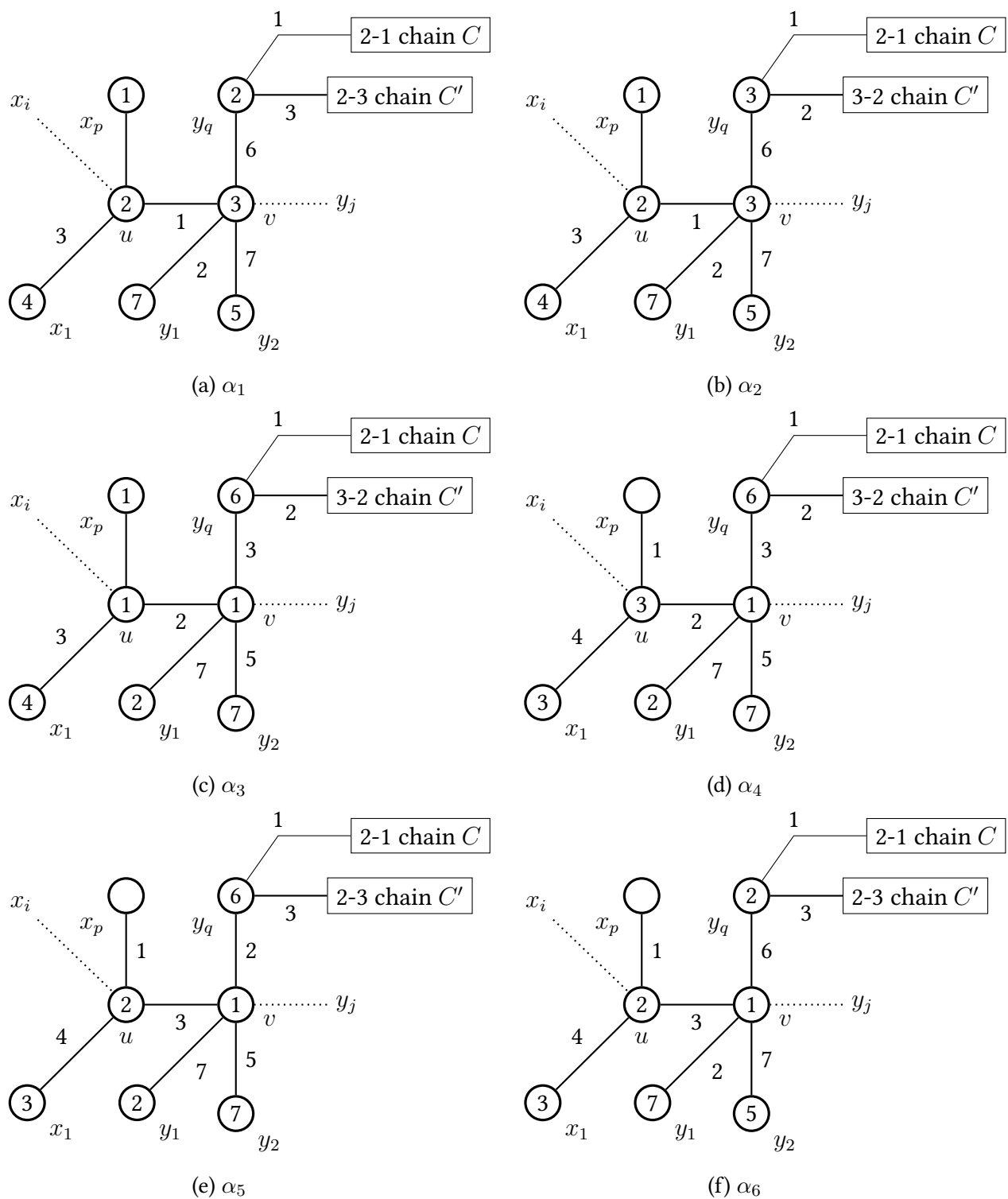

\begin{proof}[Proof of Lemma~\ref{lem:doubleloop}]
Let $X_u=(uv,ux_1,\dots,ux_p)$ be a saturated cycle in $D_u(\alpha)$, and $Y_v=(vu,vy_1,\ldots,vy_q)$ be a cycle in $D_v(\alpha)$. 
Observe that $m(v) \neq m(u)$, otherwise $X_v$ and $Y_u$ contain only the edge $uv$ and thus do not induce cycles. Note that $m(x_p)=m(y_q)=\alpha(uv)$ and by triangle-freeness $x_p\neq y_q$.

Figure~\ref{fig:double-cycle-begin}
illustrates the following argument. Since $X_u$ is saturated, the component of $K(\alpha(uv),m(u))$ containing $u$ also contains $x_p$, and thus does not contain $y_q$. 
In particular, it follows that $q \geq 2$, since by definition $\alpha(vy_1)=m(u)$ and thus $y_1$ is in the same component of $K(\alpha(uv),m(u))$ as $u$ and $x_p$, while $m(y_q)=\alpha(uv)$. 

Let $C$ be the component of $K(\alpha(uv),m(u))$ containing $y_q$. 
We note that $C$ and $X_u \cup Y_v$ are disjoint, and that neither endpoint of $C$ is incident to an edge of $X_u \cup Y_v \setminus \{vy_q\}$, as the only vertices missing colours $\alpha(uv)$ or $m_\alpha(u)$ in $X_u \cup Y_v$ are by definition $u$, $x_p$, and $y_p$, since $X_u$ and $Y_v$ induce cycles.
We consider the colouring $\alpha_{1}$ obtained from $\alpha$ by swapping $C$ (see Figure~\ref{fig:double-cycle-end}
for all the intermediate colourings used in this proof). For every $x_i$, we have $\alpha(ux_i)=\alpha_{1}(ux_i)$ and $m_\alpha(x_i)=m_{\alpha_{1}}(x_i)$; similarly for $u,v$, and every $y_j$ with $1\leq j < q$.

The sequence $X_u$ is also a cycle-inducing sequence of vertices that can be reached from $uv$ in $D_u(\alpha_{1})$. However,  $X_u$ may not saturated in $\alpha_{1}$. We distinguish the two cases. 

\begin{itemize}
    \item Assume that $X_u$ is not saturated in $\alpha_{1}$. By Lemma~\ref{lem:loop}, we have $\alpha_{1} \leftrightsquigarrow X_u^{-1}(\alpha_{1})$. By swapping $C$ for the second time (remember that $C$ and $X_u$ are disjoint, and that neither endpoint of $C$ is incident to an edge of $X_u$), we obtain $X_u^{-1}(\alpha)$, hence the conclusion.
    
    \item Assume now that $X_u$ is saturated in $\alpha_{1}$. 
    Hence the component of $K(m_{\alpha_{1}}(u),m_{\alpha_{1}}(v))$ containing $u$ also contains $v$ thus does not contain $y_q$, since $m_{\alpha_{1}}(y_q)=m_{\alpha_{1}}(u)$.
    
    Let $C'$ be the component of $K(m_{\alpha_{1}}(u),m_{\alpha_{1}}(v))$ containing $y_q$. Similarly as for $C$, we note that $C'$ and $X_u \cup Y_v$ are disjoint, and that neither endpoint of $C'$ is incident to an edge of $X_u \cup Y_v \setminus \{vy_q\}$. We consider the colouring $\alpha_{2}$ obtained from $\alpha_{1}$ by swapping $C'$. In $D_v(\alpha_{2})$, the sequence $(vu,vy_1,\ldots,vy_q)$ is the sequence of vertices of $D_v(\alpha_{2})$ that can be reached from $uv$, and it induces a path. 
    Let $\alpha_3=(uv,vy_1,\ldots,vy_q)^{-1}(\alpha_{2})$.
    By Observation~\ref{obs:path}, we have $\alpha_{2} \leftrightsquigarrow \alpha_3$. 
    Note that $\alpha_3$ assigns the colour $\alpha(uv)$ to no edge in $X_u \cup Y_v$. 
    In $D_u(\alpha_3)$, the sequence $(ux_1,\ldots,ux_p)$ is the sequence of vertices that can be reached from $ux_1$, and it induces a path.
    Let $\alpha_4$ be the colouring $(ux_1,\ldots,ux_p)^{-1}(\alpha_3)$. 
    By Observation~\ref{obs:path}, we have $\alpha_3 \leftrightsquigarrow \alpha_4$. 
    Note that in $\alpha_{4}$, we have $m_{\alpha_4}(v)=\alpha(uv)$ and $m_{\alpha_4}(u)=m_\alpha(v)$,
    with $\alpha_{4}(uv)=m_\alpha(u)$. 
    Note that there is a unique connected component of $K(m_\alpha(u),m_\alpha(v))$ containing vertices of $C'$, which is precisely $C' \cup \{uv,vy_q\}$.
    
    In the colouring $\alpha_{5}$ obtained from $\alpha_{4}$ by swapping $C' \cup \{uv,vy_q\}$, there is a unique component of $K(\alpha(uv),m_\alpha(u))$ containing vertices of $C$, which is precisely $C \cup \{vy_q\}$. 
    Moreover, in the colouring $\alpha_{5}$, the sequence $(vy_1, vy_q, vy_{q-1}, \ldots, vy_2)$ induces a cycle in $D_v$.
    The cycle is not saturated since the component of $K(\alpha(uv),m_\alpha(u))$ containing vertices of $C$ is precisely $C \cup \{vy_q\}$: since $q\geq 2$, it does not contain $y_1$. 
    We consider the colouring $\alpha_{6}$ obtained from $\alpha_{5}$ by inverting $(vy_1, vy_q, vy_{q-1}, \ldots, vy_2)$. 
    By Lemma~\ref{lem:loop}, we obtain $\alpha_{5} \leftrightsquigarrow \alpha_{6}$. 
    Note that in $\alpha_6$, the component of $K(\alpha(uv),m_\alpha(u))$ containing vertices of $C$ is precisely $C$: we swap it and obtain $\alpha \leftrightsquigarrow X_u^{-1}(\alpha)$, as desired.
\end{itemize}

\end{proof}
    
\section{The good, the bad and the ugly (edges)}

We essentially follow the outline of~\cite{casselgren}, and proceed by induction on $\Delta$. Given a $\Delta$-regular triangle-free graph $G$ that is $\Delta$-edge-colourable, we consider a $(\Delta+1)$-edge-colouring $\alpha$ and a target $\Delta$-edge colouring $\gamma$. Let $M$ be a colour class of $\gamma$, and note that $M$ is a perfect matching. We fix a colour out of $\{1,2,\ldots,\Delta+1\}$, say $1$, 
and try, through Kempe changes from $\alpha$, to assign the colour $1$ to every edge in $M$. If we succeed, we can delete $M$ and proceed by induction on $G \setminus M$ with colours $\{2,\ldots,\Delta+1\}$, noting that $\gamma$ restricted to $G \setminus M$ uses only $(\Delta(G)-1)$ colours. Let us introduce some terminology to quantify how close we are to this goal of assigning the colour $1$ to every edge in $M$.

In a given colouring, we say an edge is:
\begin{itemize}
    \item \emph{good} if it belongs to $M$ and is coloured $1$.
    \item \emph{bad} if it belongs to $M$ but is not coloured $1$.
    \item \emph{ugly} if it does not belong to $M$ but is coloured $1$.
\end{itemize}

Throughout the proof, we consider exclusively $(\Delta+1)$-colourings that can be reached from $\alpha$ through a series of Kempe changes: let us denote by $\mathcal{C}$ all such colourings. We define an order on $\mathcal{C}$ and we will prove that, in any minimal colouring, all edges of the perfect matching $M$ are coloured 1.

\begin{definition}
A colouring in $\mathcal{C}$ is \emph{minimal} if it has the fewest bad edges among all colourings in $\mathcal{C}$, and among those with the fewest bad edges, has the fewest ugly edges.
\end{definition}

Note that there may be many minimal colourings. If $m(u)=1$, we say the vertex $u$ is \emph{free}. 

\begin{lemma}\label{lem:cycleseverywhere}
In a minimal colouring, every ugly edge $vw$ is such that the sequence of vertices of $D_v$ reached from $vw$ induces a cycle.
\end{lemma}

\begin{proof}
We consider a minimal colouring $\beta$, and denote by $X_v(w)=(vw,vx_1,\ldots,vx_p)$ the sequence of vertices of $D_v(\beta)$ reached from $vw$. Suppose by contradiction that $X_v(w)$ does not induce a cycle. 
The simple yet key observation is that for every $i$, $m(x_i)\neq 1$. 

If $X_v(w)$ induces a path, we conclude immediately using Observation~\ref{obs:path}, as $X_v^{-1}(\beta,w)$ has the same number of bad edge as $\beta$, and one fewer ugly edge.

Therefore, it suffices to consider the case where $X_v(w)$ induces a comet. We let $q$ be such that $vx_p$ has an out-edge to $vx_q$ in $D_v$. 
In addition to $m(x_i)\neq 1$ for every $1\leq i\leq p$, note that $m(v)\neq 1$, as $\beta(vw)=1$. 
The colouring $\beta'$ obtained from Lemma~\ref{lem:comet} has therefore the same number of bad edges as $\beta$, and fewer ugly edges. Since $\beta' \leftrightsquigarrow \beta$, this contradicts the minimality of $\beta$.
\end{proof}

By considering the last element of a sequence reached from an ugly edge, Lemma~\ref{lem:cycleseverywhere} yields the following statement, whose proof appeared in~\cite{casselgren} but which we state somewhat differently. 

\begin{corollary}[\cite{casselgren}]\label{cor:fan}
In a minimal colouring, both endpoints of an ugly edge have a free neighbour.
\end{corollary}

As we shall see, a consequence of Corollary~\ref{cor:fan} 
together with the regularity assumption
is that,
in a minimal colouring, there are bad edges with a free endpoint (unless there is no bad edge at all). 
These are central to the argument\footnote{In~\cite{casselgren}, this allows us to assume case A happens, avoiding case B entirely.}. 
Let us now prove a small observation and then proceed with the core of the proof.

\begin{observation}\label{obs:ugly-from-bad}
In any minimal colouring $\beta$, every bad edge is incident to an ugly edge.
\end{observation}

\begin{proof}
Let $xy$ be a bad edge. If $m(x)=m(y)=1$, we can swap the (single-edge) component of $K(1,\beta(xy))$ containing $xy$ and have one fewer bad edges, a contradiction to the minimality of $\beta$. We derive that $xy$ is incident to some edge $e$ satisfying $\beta(e)=1$.
Then $e$ is necessarily ugly, as $xy \in M$ and $M$ is a matching. 
\end{proof}

\begin{proof}[Proof of Theorem~\ref{th:main}]
Let $\beta$ be a minimal colouring. If there is no bad edge, then $M$ is monochromatically coloured, as desired. Therefore, we assume that there is a bad edge which, by Observation~\ref{obs:ugly-from-bad}, is incident to an ugly edge $e$.
By Corollary~\ref{cor:fan} applied to $e$, there exists some free vertex $u$ (adjacent to an endpoint of $e$). 

Let $v$ be such that $uv \in M$,
note that since $u$ free, $uv$ is bad, and is thus incident to an ugly edge by Observation~\ref{obs:ugly-from-bad}. Since $u$ cannot be incident to an ugly edge (it is free), there is some vertex $w \in N(v)$ such that $vw$ is ugly. We denote by $X_v$ the sequence of vertices of $D_v$ reached from $vw$, and by $Y_w$ the sequence of vertices of $D_w$ reached from $vw$.

By Lemma~\ref{lem:cycleseverywhere}, we obtain immediately that $X_v$ induces a cycle in $D_v$, and $Y_w$ induces a cycle in $D_w$. 
By Lemmas~\ref{lem:loop} or~\ref{lem:doubleloop}, we derive that $\beta \leftrightsquigarrow Y_w^{-1}(\beta)$. Note that $Y_w^{-1}(\beta)$ has at most as many bad and ugly edges as $\beta$. 

By triangle-freeness, $u$ and $w$ are not adjacent and so $uw$ does not appear in $Y_w$. 
Thus $m_{Y_w^{-1}(\beta)}(v)=1=m_{Y_w^{-1}(\beta)}(u)$. 
We swap the (single-edge) component of $K(1,\beta(uv))$ containing the edge $uv$, and obtain a colouring with fewer bad edges, a contradiction.
\end{proof}

\section{Conclusion}

The main result of this paper is based on the Vizing's fans; in his proof he only needs to handle the case of paths and comets (as described in Section~\ref{sec:fans}), our main contribution is to extend this argument to non saturated cycles and double cycles. The proof yields an algorithm that, given a target optimal colouring, computes a polynomial transformation sequence in polynomial-time. 
The method could be pushed further, however the length and the complexity of the proof would increase exponentially. The general case for multigraphs still seems a challenging question, and an additional argument is probably needed to solve it.

Moreover, even some special cases are of their own interest. For instance, it's still unclear how to handle the case of cliques, as an induction on the chromatic index is not possible for this class of graphs.
\paragraph*{Acknowledgements.}

We gratefully acknowledge support from Nicolas Bonichon and the Simon family for the organisation of the $5^{\textrm{th}}$ Pessac Graph Workshop, where a preliminary part of this research was done. We thank Caroline Brosse, Vincent Limouzy, Carole Muller and Lucas Pastor for extensive discussions around this topic. Last but not least, we thank Peppie for her unwavering support during the work sessions, close or from afar.

\bibliography{kempe}

\end{document}